\documentclass[11pt, a4paper, twoside, reqno]{amsart}

\usepackage[english]{babel}

\usepackage[utf8]{inputenc}
\usepackage[T1]{fontenc}
\usepackage{lmodern}
\usepackage{verbatim}

\usepackage{amsmath}
\usepackage{amssymb}
\usepackage{amsthm}
\usepackage{thmtools, thm-restate}

\usepackage{csquotes}

\usepackage{mathtools}
\usepackage{stmaryrd}

\usepackage{wrapfig}
\usepackage{subcaption}
\usepackage{multirow}

\usepackage{sidenotes}
\usepackage{enumitem}
\usepackage{verbatim}
\usepackage{wrapfig}
\usepackage{makecell}
\usepackage{xcolor}
\usepackage{todonotes}

\usepackage{hyperref}

\usepackage[numbers,sort&compress]{natbib}
\usepackage{tikz}

\DeclareMathOperator*{\pos}{pos}
\DeclareMathOperator*{\lin}{lin}

\DeclareMathOperator*{\conv}{conv}
\DeclareMathOperator*{\aff}{aff}

\DeclareMathOperator*{\inter}{int}

\newcommand{\R}{\mathbb{R}}

\newcommand{\N}{\mathbb{N}}

\mathtoolsset{showonlyrefs}

\newtheorem{theorem}{Theorem}[section] 
\newtheorem{corollary}[theorem]{Corollary}
\newtheorem{proposition}[theorem]{Proposition}
\newtheorem{lemma}[theorem]{Lemma}

\newtheorem{remark}[theorem]{Remark}
\newtheorem{theoman}{Theorem}[section]

\theoremstyle{definition}

\newtheorem{definition}[theorem]{Definition}

\theoremstyle{remark}

\newtheorem{example}[theorem]{Example}

\newcommand{\Sn}{\mathbb{S}^{n-1}}
\newcommand{\Pscc}{P_{\texttt{scc}}(U)}

\newcommand{\Pcone}{C_{\cv}(U)}

\newcommand{\clPcone}{\mathrm{cl}(\mathrm{conv}({C_{\cv}(U)}))}

\newcommand{\Utriangle}{U_{\Delta}}
\newcommand{\Usquare}{U_{\square}}

\newcommand{\vol}{\mathrm{vol}\,}
\newcommand{\sph}{\mathbb{S}}
\newcommand{\scc}{{\tt scc}} 
\newcommand{\cv}{{\tt cv}}

\newcommand{\ov}{\overline}
\DeclareMathOperator{\V}{V}
\DeclareMathOperator{\relint}{relint}

\let\emptyset\varnothing

\numberwithin{equation}{section}

\begin{document}

\title{On the discrete logarithmic Minkowski problem in the plane}
\author{Tom Baumbach}

\address{Technische Universität Berlin, Institut für Mathematik, Sekr.~MA4-1, Straße des 17.~Juni 136, 10623 Berlin, Germany}
\email{baumbach@math.tu-berlin.de}

       \begin{abstract}
         The paper characterizes the convex hull of the closure of the cone-volume set $C_\cv(U)$, consisting of all cone-volume vectors of polygons with outer unit normals vectors contained in $U$, for any finite set $U \subseteq \R^2, \pos(U) = \R^2$. We prove that this convex hull has finitely many extreme points by providing both a vertex representation as well as a half space representation. As a consequence, we derive new necessary conditions, which depend on $U$, for the existence of solutions to the logarithmic Minkowski problem in $\R^2$.
       \end{abstract}


   \maketitle

   \section{Introduction}
The setting for this paper is the $n$-dimensional Euclidean space $\R^n$. For two vectors $x,y \in \R^n$ we denote by $\langle x,y \rangle$ the standard scalar product of $x$ and $y$, and  $\| x \| = \sqrt{\langle x,x \rangle}$  denotes the associated Euclidean norm; $\sph^{n-1}=\{x\in\R^n : \| x \|=1\}$ is the $(n-1)$-sphere.  The convex hull of a non-empty set $M\subset\R^n$ is denoted by $\conv M$, and if $M$ is finite then $\conv M$ is called a polytope; expressing a polytope as $P = \conv(M)$ is referred to as its $\mathcal{V}$-representation. By a result attributed to Minkowski and Weyl, $P\subset\R^n$ is a polytope if and only if 
    \begin{equation*}
       P=P(U,b)=\{x\in\R^n : U^\intercal x\leq b\} = \bigcap_{i = 1}^m H^-(u_i,b_i)
    \end{equation*}
    for a  matrix $U=(u_1,\dots,u_m)\in(\Sn)^m$ with $\pos ( U)=\R^n$ and $b\in\R^m$. This is called the $\mathcal{H}$-representation of the polytope. 
    Here $\pos (U)$ means the positive hull, i.e., the set of all non-negative linear combinations of the column vectors $u_1,\dots,u_m\in\Sn$ of $U$.  Apparently, we may assume that the column vectors are pairwise different, and therefore we set
    \begin{equation*}
      \mathcal{U}(n,m)=\left \{U=(u_1,\dots,u_m)\in(\Sn)^m : \pos (U)=\R^n, u_i\ne u_j, i \ne j\right\}. 
    \end{equation*}

    For $1\leq i\leq m$ let
\begin{equation*} 
  F_i(b)=F(u_i,b)=P\cap\{x\in\R^n : \langle u_i,x\rangle =b_i\}
\end{equation*}
  which is always a face of $P$ and might be empty. If $\dim F_i(b)=\dim P-1$, $F_i(b)$ is called a facet of $P$. For $M\subset \R^n$ we denote by $\vol(M)$ its volume, i.e., its $n$-dimensional Lebesgue measure. If $M$ is contained in a
$k$-dimensional plane $A$, $\vol_k(M)$ refers to the $k$-dimensional Lebesgue measure with respect to $A$.


We will mainly assume that $b\geq 0$. This implies $0\in P$, and if $b > 0$ then  $0 \in \inter P$, i.e., $0$ is an interior point of $P$, and so $\dim P=n$.   
If  $F_i(b)$ is a facet of $P(U,b)$, 
then $\frac{1}{n}b_i\vol_{n-1}(F_i(b))$ is the volume of the cone (pyramid)
$\conv(\{0\}\cup F_i(b))$. As for $U\in\mathcal{U}(n,m)$, $P(U,b)$ is the interior-disjoint union of all these cones we can write
\begin{equation*}
            \vol(P(U,b))=\frac{1}{n}\sum_{i=1}^m b_i\vol_{n-1}(F_i(b)).
\end{equation*}
The cone-volume vector $\gamma(U,b)$ of $P(U,b)$ has as entries the volumes of the cones $\conv(\{0 \} \cup F_i)$
\begin{equation*}
    \gamma(U,b) = \left(\frac{1}{n}b_1\vol_{n-1}(F_1(b)),\dots,\frac{1}{n}b_m\vol_{n-1}(F_m(b)) \right),
\end{equation*}
and the cone-volume set is $C_{\cv}(U) = \{ \gamma(U,b) : b \in \R_{\geq 0}^m, \vol_n(P(U,b)) = 1 \}$, cf. \cite{baumbach2025polynomialinequalitiesconevolumespolytopes}. 

For such a $P=P(U,b)$, $\dim P=n$, we consider its cone-volume measure $\V_P$ which is  
the finite non-negative Borel measure $\V_P:\mathcal{B}(\sph^{n-1})\to\R_{\geq 0}$, where $\mathcal{B}(\sph^{n-1})$ contains all measurable subsets of $\Sn$,          
 given by
\begin{equation*}
  \V_P(\eta) =\sum_{i=1}^m \frac{b_i}{n}\vol_{n-1}(F_i(b))\,\delta_{u_i}(\eta) =\sum_{u_i\in\eta} \frac{b_i}{n}\vol_{n-1}(F_i(b)). 
\end{equation*}
Here $\eta\subseteq \sph^{n-1}$ is a Borel set and  $\delta_{u_i}(\cdot)$ is the Dirac measure in $u_i$, i.e., $\delta_{u_i}(\eta)=1$ if $u_i\in\eta$, otherwise it is $0$.  

The discrete logarithmic Minkowski (existence) problem introduced by Böröczky, Lutwak, Yang and Zhang \cite{boroczky2013logarithmic_JAMS} asks for necessary and sufficient conditions such that a finite discrete Borel measure  
\begin{equation}
  \mu:\mathcal{B}(\sph^{n-1})\to\R_{\geq 0}, \quad \mu(\eta)=\sum_{i=1}^m \gamma_i\,\delta_{u_i}(\eta)
\label{eq:measure}  
\end{equation}
with $u_i\in\sph^{n-1}$, $\gamma_i> 0$, is the cone-volume measure of a polytope.  We will denote such a measure also by  $\mu(U,\gamma)$, where $\gamma\in\R^m_{>0}$ is the vector with entries $\gamma_i$.

This discrete problem can be extended to the continuous setting, i.e., to the space of all convex bodies and  
the corresponding general logarithmic Minkowski problem is a corner stone of modern convex geometry. For its history, relevance and impact we refer to \cite{HuangYangZhang2025, BoeroeczkyLutwakYangEtAl2025, Stancu2016, BoeroeczkyHenk2016a, LiuSunXiong2021, Zhu2014, Boeroeczky2023, BoeroeczkyHegedusZhu2016, ChenLiZhu2019, boroczky2013logarithmic_JAMS} and to the references within. Here we will only focus on the discrete setting. 

In order to study the discrete logarithmic Minkowski problem in the plane, we consider the cone-volume set $C_\cv(U)$, for $U \in \mathcal{U}(2,m)$. 

The existence of solutions for the discrete planar logarithmic Minkowski problem has been investigated in special situations by Stancu \cite{stancu2002discrete,stancu2003number}. Specifically, Stancu studied the problem for special outer unit normal sets \(U\) and analyzed the number of solutions \cite{stancu2003number}. Pollehn \cite{Subspace_Concentration_of_Geometric_Measures} computed the cone-volume set for trapezoids. Liu et al. \cite{LiuLuSunEtAl2024} introduced a new necessary condition in the planar case for centered polygons and examined the solutions for quadrilaterals. 

For the even logarithmic Minkowski problem, we consider sets $U = U^s \in \mathcal{U}(n,2m)$, such that $u_i = - u_{i+m}$ for $i = 1,\dots,m$, and restrict attention to origin symmetric polytopes $P = P(U^s,b^s)$, with a symmetric right-hand side $b^s_i = b^s_{i+m}, i = 1,\dots,m$. The corresponding symmetric cone-volume set is $$C_{\cv}^s(U^s) = \{ \gamma(U^s,b^s) : b^s \in \R_{\geq 0}^{2m}, \vol_n(P(U^s,b^s)) = 1 \}.$$ Stancu \cite{stancu2002discrete} provided a complete characterization of the symmetric cone-volume set in the plane. In the following, $\relint(M)$ denotes the relative interior of $M\subseteq\R^n$, i.e., the set of interior points with respect to the ambient space given by $\aff M$, the affine hull of $M$.

\begin{theoman}\cite[Theorem 1.2.]{stancu2002discrete}
\label{thm:Stancu_Symmetric}
    Let $U^s \in \mathcal{U}(2,2m)$ and $m > 2$. Then it holds 
  \begin{equation*}
    C^s_\cv(U^s) \cap\R^{2m}_{>0}=\relint \left(P_\scc (U^s) \right)\cap 
    \left\{ x \in \R^{2m} : x_{i } = x_{m+ i},\,1\leq i\leq m \right\}.
 \end{equation*}
\end{theoman}
In the preliminaries, we will recall the definition of the polytope $P_\scc(U)$, for $U\in\mathcal{U}(2,m)$, introduced in \cite{baumbach2025polynomialinequalitiesconevolumespolytopes}, which is called the subspace concentration polytope (of $U$). Up to scaling, the polytope $P_\scc(U)$ is (just) the matroid base polytope of the set of column vectors of $U$.

Theorem \ref{thm:Stancu_Symmetric} was later generalized in the groundbreaking paper \cite{boroczky2013logarithmic_JAMS}. It was shown that an even finite positive Borel measure $\mu(U,\gamma)$ is the cone-volume measure of an origin symmetric polytope $P(U,b)$ if and only if $\mu(U,\gamma)$ satisfies the subspace concentration condition (\scc), introduced by Böröczky et al. \cite{boroczky2013logarithmic_JAMS}.   A finite  discrete Borel measure $\mu=\mu(U,\gamma): \mathcal{B}(\sph^{n-1})\to\R_{\geq 0}$ with $U\in\mathcal{U}(n,m)$, $\gamma >0$, is said to satisfy the \scc\, if
\begin{enumerate}
\item for every proper linear subspace $L\subset\R^n$ it holds 
\begin{equation}
 \label{eq:scc1} 
           \mu(L)=  \sum_{u_i\in L} \gamma_i \leq \frac{\dim L}{n}\mu(\sph^{n-1}),
  \end{equation}
\item   equality holds in \eqref{eq:scc1} if and only if there exists a subspace $\ov L$ complementary to  $L$ such that $\{u_1,\dots,u_m\}\subset L\cup\ov L$.   
\end{enumerate}   

This characterization can be reformulated using the cone-volume set (see also \cite[Thm.~4.7]{LiuSunXiong2024b}). For $U^s\in\mathcal{U}(n,2m)$, we have
  \begin{equation*}
    C^s_\cv(U^s) \cap\R^{2m}_{>0}=\relint \left(P_\scc (U^s) \right)\cap  
    \left\{ x \in \R^{2m} : x_{i } = x_{m+ i},\,1\leq i\leq m \right\}.
 \end{equation*}


In the non-symmetric case the following equality holds. Consider a set $U\in\mathcal{U}(n,m)$. Then $C_\cv(U)=P_\scc(U)$ if and only if  $m=2n$ and up to renumbering we have $u_{n+i}=-u_i$, $1\leq i\leq n$ (\cite[Thm.~2.11]{baumbach2025polynomialinequalitiesconevolumespolytopes}).

\vspace{\baselineskip}

For non-symmetric discrete measures in $\mathbb{R}^2$, Stancu \cite{stancu2002discrete} provided some sufficient conditions for the existence. 

\begin{theoman}\cite[Theorem 1.1.]{stancu2002discrete}
\label{StancuSufficientConditions}
    Let $U \in \mathcal{U}(2,m)$ and $\mu: \mathcal{B}(\mathbb{S}^1) \rightarrow \mathbb{R}_{\geq 0}, \linebreak \mu(\eta) = \sum_{i = 1}^m \gamma_i \delta_{u_i}(\eta)$ with $\gamma_i > 0$ for $i =1,\dots,m$. Assume that one of the following holds:
    \begin{enumerate}
        \item[i)] $m \geq 4$ and $U$ consists of pairwise linearly independent vectors.
        \item[ii)] $m > 4$ and $U$ contains, at least, two linearly dependent vectors. For any $j,k$ with $u_j = - u_k$, we have 
        \begin{equation*}
            \gamma_j + \gamma_k < \sum_{i \neq j,k} \gamma_i.
        \end{equation*}
        \item[iii)] $m = 4$, $U$ contains a unique pair of opposite vectors, $u_1 = -u_3$, and $\gamma_1 + \gamma_3 < \gamma_2 + \gamma_4$.
    \end{enumerate}
    Then there exists a polygon $P = P(U,b)$ such that $\mu = V_P$.
\end{theoman}

By definition, for $U\in\mathcal{U}(2,m)$, the cone-volume set $C_\cv(U)$ is a subset of $\{x\in\R^{m}: x\geq 0, \, x_1+x_2+\dots + x_m=1\}$ and it can be that large by Theorem \ref{StancuSufficientConditions} i). This was later generalized by Zhu \cite[Thm.]{Zhu2014}.
Let  $U\in\mathcal{U}(n,m)$ be in general position, i.e., any $n$ columns of $U$ are linearly independent. Then 
  \begin{equation*}
    \begin{split} 
      C_\cv(U) \cap\R^m_{>0} &= \left\{x\in\R^{m}:  x>0,\, x_1+x_2+\dots + x_{m}=1\right\}\\ & = \conv\{e_1,\dots,e_m\}\cap\R^m_{>0} .
   \end{split}   
  \end{equation*}

By a result of Chen et al. \cite{ChenLiZhu2019}, we also have the following inclusion, which is a generalization of Theorem \ref{StancuSufficientConditions} ii).
Let $U\in\mathcal{U}(n,m)$. Then it holds
  \begin{equation*}
    C_\cv(U)\cap\R^m_{>0} \supseteq \relint P_\scc(U).
  \end{equation*}


    

In the general setting, we know that the cone-volume set $C_\cv(U)$ is a semialgebraic set \cite[Theorem 1.3]{baumbach2025polynomialinequalitiesconevolumespolytopes} and we could theoretically compute necessary and sufficient conditions in the (non-symmetric) discrete logarithmic Minkowski problem. However, this approach is not very practical, since we need to make use of the Tarski-Seidenberg principle \cite[Theorem 2.2.1]{bochnak2013real}. 




Further, in general the cone-volume set $C_\cv(U)$ is neither convex nor closed \cite[Proposition 2.9]{baumbach2025polynomialinequalitiesconevolumespolytopes}. Therefore, we study the extreme points of the convex hull of the closure of the cone-volume set $C_\cv(U)$, denoted by $\clPcone$, where we restrict ourselves to the sets $U \in \mathcal{U}(2,m)$. 

For the characterization of the extreme points of the set $\clPcone$ we consider the following subsets of $U$.

\begin{definition}
\label{Definition:OuterNormalsInUsquareAdjacentFacets}
    Let $U \in \mathcal{U}(2,m)$. Then we define the following three sets
    \begin{align}
        \Utriangle & \coloneqq \{ u \in U : \text{ there are vectors } v,w \in U \text{ s.t. } \pos\{u,v,w\} = \R^2 \}, \\
        \Usquare & \coloneqq U \setminus \Utriangle \text{ and }   \\
        U_{\square,u} & \coloneqq \{ x,y \in U :  \pos\{u,-u,x,y\} = \R^2 \} \text{ for } u \in \Usquare.
    \end{align}
\end{definition}

We are ready to state the $\mathcal{V}$-representation of $\clPcone$.

\begin{restatable}{theorem}{VRepresentation}
    \label{thm:VRepresentationOfclPconePlanar}
    Let $U \in \mathcal{U}(2,m)$. The vertices of $\clPcone$ are given by 
    \begin{equation}
         \bigcup_{i : u_i \in \Utriangle} \{ e_i \} \cup \bigcup_{(i,j) : u_i \in \Usquare, u_j \in U_{\square,u_i} } \left\{ \frac{1}{2}(e_i + e_j) \right\} .
    \end{equation}
\end{restatable}

Further, we also provide a $\mathcal{H}$-representation of $\clPcone$.

\begin{restatable}{theorem}{HRepresentation}
\label{thm:HRepresentationOfclPconePlanar}
    Let $U \in \mathcal{U}(2,m)$. Then
    \begin{align}
    \clPcone = & \left( \bigcap_{i: u_i \in \Utriangle} H^-(e_i,1) \right) \cap \left( \bigcap_{(i,j) : u_i \in \Usquare, - u_j = u_i} H^- \left(e_i + \frac{1}{2} e_j, \frac{1}{2} \right) \right) \\
     & \cap \left( \bigcap_{(i,k) : u_i \in \Usquare, u_k \in U_{\square,u_i}} H^-(e_i + e_k , 1)  \right) \cap \left( \bigcap_{i = 1}^m H^-(-e_i,0) \right) \\ 
     & \cap H(e_1 + \dots + e_m, 1),
\end{align}
\end{restatable}
In general, the $\mathcal{H}$-representation of $\clPcone$ is redundant: not all inequalities define facets (cf.\ Proposition \ref{Prop:VRepresentationTrapezoid} and the subsequent discussion).

The paper is structured as follows. In Section \ref{Section:Preliminaries} we recall the definition of the subspace concentration polytope $P_\scc(U)$ and the cone-volume set $C_\cv(U)$ as well as their basic properties. In Section \ref{SectionPlanarCase} we analyze the structure of the subsets $\Utriangle$, $\Usquare$ and $U_{\square,u}$. Using the results about the structure of these sets, we develop the theory in order to characterize the closure of the convex hull of the cone-volume set for sets $U \in \mathcal{U}(2,m)$, which gives us all the necessary linear inequalities for the logarithmic Minkowski problem.

   \section{Preliminaries}
\label{Section:Preliminaries}
This paper builds upon the work presented in \cite{baumbach2025polynomialinequalitiesconevolumespolytopes}. For the sake of completeness, we revisit its preliminaries, focusing only on the planer case.
In this section, we recall the basic properties of the subspace concentration polytope as well as the cone-volume set. 

Additionally, we restrict our attention to finite sets of \( \mathbb{R}^2 \) and \( \mathbb{S}^1 \), ensuring a well-defined combinatorial and geometric context for all discussions and results.


For a set $U \in \mathcal{U}(2,m)$ we define the subspace concentration set, as in \cite{baumbach2025polynomialinequalitiesconevolumespolytopes}.  
With $S \subseteq U$, we refer to a subtuple of columns in $U$ and by $\mathrm{rk}(S)$ we denote the dimension of the linear span of $S$ with respect to the columns.

There are two important sets regarding the matroid polytope, corresponding to the two different cases in $\scc$. Both of these sets consist of subsets of $U$:
\begin{enumerate}[label = \roman*)]
    \item $\mathcal{L}(U) \coloneqq \Bigl\{ S \subseteq U :  \mathrm{rk}(S) = 1 \text{ and } U \cap \lin(S) = S \Bigl\}$
    \item $\mathcal{F}(U) \coloneqq \Bigl\{ S \in \mathcal{L}(U) : \lin(S) \cap \lin(U \setminus S) = \{ 0 \} \Bigl\}$.
\end{enumerate}
Since we only work with sets $U \in \mathcal{U}(2,m)$, we have the following characterization.

\begin{proposition}
    Let $U \in \mathcal{U}(2,m)$. Then $\mathcal{L}(U) \neq \emptyset$ if and only if $U = \{ \pm u, \pm v \}$ for some $u,v \in \mathbb{S}^1$.
\end{proposition}

\begin{proof}
     First, it is clear that $\{ \pm v \} \in \mathcal{L}(U)$ if $U = \{ \pm u \} \cup \{ \pm v \}$.
    
    For the other direction, consider a set $U$ such that $S \in \mathcal{L}(U)$. Then there exists vectors $u,v \in \mathbb{S}^1$ such that $S \subseteq \{ \pm u \}$ and $U \setminus S \subseteq \{ \pm v\}$. Since it holds 
    \begin{equation*}
        \R^2 = \pos(U),
    \end{equation*}
    we get that $U = \{ \pm u, \pm v \}$. 
\end{proof}

The subspace concentration polytope is defined as the set 
\begin{align}
    P_\scc(U) \coloneqq \, & \Bigl\{ x \in \R^m : x_i \geq 0, \sum_{i = 1}^m x_i = 1, \\
    & \sum_{i : u_i \in \lin S} x_i \leq \frac{ rk(S)}{n}, \text{ for all } S \in \mathcal{L}(U) \Bigl\},
\end{align}
which has a close connection with the subspace concentration condition and the cone-volume set, cf. \cite{boroczky2013logarithmic_JAMS,BoeroeczkyHegedusZhu2016,Zhu2014,ChenLiZhu2019}.
The subspace concentration polytope $P_\scc(U)$ is a special case of matroid polyhedra \cite{chatelain2011matroid}, where the matroid consists of $U$ and $\mathcal{I}$, the set that contains all linear independent subsets of $U$.

The set $P_\scc(U)$ is a polytope, meaning $P_\scc(U)$ can be expressed with half spaces, directly coming from its definition, or vertices. The $\mathcal{V}-$representation of $P_\scc(U)$ was already established in the papers \cite{cunningham1984testing,chatelain2011matroid}, which we are going to recall 
\begin{equation}
    P_\scc(U) = \conv\left\{ \frac{1}{2} \sum_{u_i \in B} e_i : B \subseteq U \text{ is a linear basis of } \R^2\right\}.
\end{equation}
Further, the dimension of the matroid polytope was computed and proven in \cite{FeichtnerSturmfels2005}.

Next, we consider the cone-volume set in the plane. For $U\in\mathcal{U}(2,m)$ and  $b\in \R^m_{\geq 0}$ the cone-volume vector is
\begin{equation*}
  \gamma(U,b)=\frac{1}{2}\Big(b_1\vol_{1}(F(u_1)),\dots, b_m\vol_{1}(F(u_m))\Big)^\intercal\in \R^m_{\geq 0}. 
\end{equation*}
Observe that some of its entries might be zero, if $F_i(b)$ is not an edge of the polygon $P(U,b)$ or $b_i=0$. The cone-volume set is 
\begin{equation*}
  C_\cv(U)=\left\{ \gamma(U,b) : b\in\R^m_{\geq 0} \text{ and } \vol_2(P(U,b))=1\right\}.
\end{equation*} 
Any cone-volume vector $\gamma(U,b)$ of a $2$-dimensional polygon of the type $P(U,b)$ is up to scaling to volume 1 an element of $C_\cv(U)$ as   
\begin{equation}
 \frac{1}{\vol_2(P(U,b))}\gamma(U,b) = \gamma\left(U, \left(\vol_2(P(U,b)) \right)^{-1/2}b\right)\in C_\cv(U).
 \label{eq:scaling} 
\end{equation}

The subspace concentration polytope $P_\scc(U)$ as well as the cone-volume set $C_\cv(U)$ can be decomposed by considering so called \textit{separators} of $U$, which are special elements of $\mathcal{F}(U)$. 

\begin{definition}
    \label{Definition:IrreducibleSets}
    Let $U \in \mathcal{U}(2,m)$. We call the set $U$ irreducible if $\mathcal{F}(U) = \emptyset$, otherwise reducible.
\end{definition}

Furthermore, for a separator $S \in \mathcal{F}(U)$ we must have $S = \lin(S) \cap U$, which is also known as $S$ is closed in matroid language.
Observe that the number of separators of a given set $U$ is finite, and thus the number of irreducible separators is finite as well. Using the separators, we obtain a much simpler representation of the subspace concentration polytope $P_\scc(U)$.

From matroid theory, we deduce the following statement \cite[Theorem 6.81]{Aigner1997}.

\begin{proposition}\cite[Proposition 2.3 + 2.7]{baumbach2025polynomialinequalitiesconevolumespolytopes}
\label{PropositionSeparatorsPsccAndConeVolumesDimension}
Let $U \in \mathcal{U}(2,m)$ be reducible. Then, up to renumbering, it holds 
\begin{equation*}
   P_\scc(U)= \left\{ x \in \R_{\geq 0}^4 : x_1 + x_2 = \frac{1}{2} \right\} \, \cup \, \left\{ x \in \R_{\geq 0}^4 : x_3 + x_4 = \frac{1}{2} \right\},
 \end{equation*}
 as well as 
 \begin{equation*}
   C_\cv(U)= \left\{ \gamma \in \R_{\geq 0}^4 : \gamma_1 + \gamma_2 = \frac{1}{2} \right\} \, \cup \, \left\{ \gamma \in \R_{\geq 0}^4 : \gamma_3 + \gamma_4 = \frac{1}{2} \right\}.
 \end{equation*}
\label{prop:sccdirect}
\end{proposition}  

Therefore, the reducible sets are of a very special form, which we use in Section \ref{SectionPlanarCase}.

Finally, we remark that $P_\scc(U)$ and $C_\cv(U)$ are in particular
linear invariant which will be used later on.
\begin{lemma}\cite[Proposition 2.1]{baumbach2025polynomialinequalitiesconevolumespolytopes}
\label{LemmaPsccAndPconeWithLinearMap}Let $U$ as above and let $A\in\mathrm{GL}(2,\R)$.  Then
  $P_\scc(AU)=P_\scc(U)$ and $C_\cv(U) = C_\cv(AU)$.
\label{prop:invariant}  
\end{lemma}

    \section{Planar Case}
\label{SectionPlanarCase}

In this section we focus on the planar case, i.e. we consider the case $n = 2$. The convex body $\clPcone$ is shown to be in fact a polytope and a complete $\mathcal{V}$-representation and $\mathcal{H}$-representation is given. 

First, we will look at trapezoids, where the cone-volume set was completely characterized in \cite[Theorem 2.14]{Subspace_Concentration_of_Geometric_Measures}, providing a complete characterization of the setting in Theorem \ref{StancuSufficientConditions} iii).

\setcounter{theoman}{2}

\begin{theoman}
\label{Theorem:ConeVolumeSetTrapezoids}
    Let $\mu$ be a non-zero, finite Borel measure on $\mathbb{S}^1$ supported on pairwise distinct and counterclockwise ordered unit vectors $u_1,u_2,u_3,u_4 \in \mathbb{S}^1$.
    Further, we assume that $u_1 = -u_3$ and there is an open hemisphere $u_3 \in \omega \subset \mathbb{S}^1$ such that $u_2,u_4 \not \in \omega$ and $u_2 \not = -u_4$.
    
     The measure $\mu$ is a cone-volume measure, i.e. there exists a polygon $P$ with outer normal set contained in $U = \{ u_1,u_2,u_3,u_4 \}$ such that $V_P = \mu$, if and only if either 
    \begin{enumerate}[label = (\roman*) ]
        \item $\mu(u_1) + \mu(u_3) < \mu(u_2) + \mu(u_4)$, or 
        \item $\mu(u_1) + \mu(u_3) \geq \mu(u_2) + \mu(u_4) \geq 2 \sqrt{\mu(u_1) \mu(u_3)}$ and $\mu(u_1) < \mu(u_3)$.
    \end{enumerate}
\end{theoman}

\medskip
We will examine the situation of parallelograms and trapezoids in great detail in order to be able to provide $\mathcal{V}$-representations and $\mathcal{H}$-representations in the planar case for the convex hull $\clPcone$. This allows us to get a better understanding of the cone-volume sets of triangles, trapezoids, and parallelograms, which is essential in order to understand the extreme points of the convex hull of general cone-volume sets. 

First, we consider the situation of Theorem \ref{Theorem:ConeVolumeSetTrapezoids} and want to find a $\mathcal{V}$- and $\mathcal{H}$-representation of the set $\clPcone$.
From these inequalities, it is clear that the cone-volume set $\Pcone$ is not convex, as the inequalities describing the cone-volume set involve inequalities with the square root. If we draw it in the 3-dimensional space, it can be partitioned into two subsets, one which corresponds to $P_\scc(U)$ and another one, which fulfills the square root inequality. In order to visualize the cone-volume set $\Pcone$, which has dimension $3$, we projected it into $\R^3$. 
The subset contained in $\Pcone$ corresponding to inequality (i) contains the subspace concentration polytope $P_\scc(U)$, implying $\Pscc \subset \Pcone$ and thus $P_\scc(U) \subset \clPcone$.

\begin{figure}[h]
        \centering
        \begin{subfigure}{0.45\textwidth}
        \center
            \includegraphics[scale = 0.45]{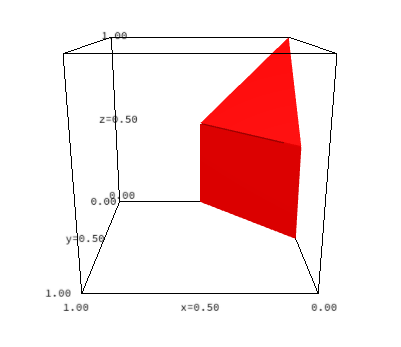}
        \caption{Case (i) of Theorem \ref{Theorem:ConeVolumeSetTrapezoids}}
       
        \end{subfigure}
        \begin{subfigure}{0.45\textwidth}
        \center
            \includegraphics[scale = 0.45]{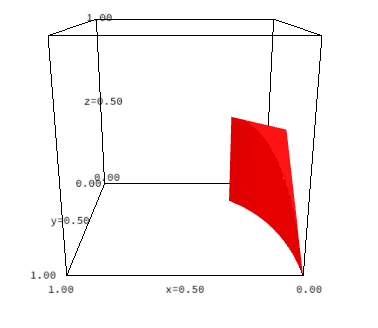}
            \caption{Case (ii) of Theorem \ref{Theorem:ConeVolumeSetTrapezoids}}
        \end{subfigure}
        \caption{The cone-volume set $\Pcone$ is the union of the subsets represented in (A) and (B). Let $\gamma \in C_\cv(U)$. The $x-$axis corresponds to $\gamma_{1}$, the $y-$axis to $\gamma_{3}$ and the $z-$axis to $\gamma_2$. The corresponding vector in $\Pcone$ is given via the formula $(\gamma_1,\gamma_2,\gamma_3,1-\gamma_1-\gamma_2-\gamma_3)$.}
        \label{FigurePconeTrapezoid}
\end{figure}

We provide visual proofs of the following propositions to build intuition; fully rigorous proofs could be derived from the ideas in this section, though they would require extensive computations.

\begin{proposition}
\label{Prop:VRepresentationTrapezoid}
    Let $U = \{ u_1,u_2,u_3,u_4 \} \subset \mathbb{S}^1$ be a counterclockwise ordered set of unit vectors with $u_1 = - u_3$ and $\omega \subset \mathbb{S}^1$ an open hemisphere such that $U \cap \omega = \{ u_3 \}$ and $u_2 \neq u_4$. Then 
    \begin{equation*}
        \clPcone = \conv\left\{e_2,e_3,e_4,\frac{1}{2}(e_1 + e_2),\frac{1}{2}(e_1 + e_4)\right\}
    \end{equation*}
    and 
    \begin{align*}
        \clPcone = & \bigcap_{i = 2}^4 H^-(e_i,1) \cap  H^-\left(e_1 + \frac{1}{2} e_3,\frac{1}{2}\right) \cap H^-(e_1 + e_2,1) \\ \cap & \, H^-(e_1 + e_4,1) \cap H(e_1 + e_2 + e_3 + e_4,1) \cap \bigcap_{i = 1}^4 H^+(e_i,0).
    \end{align*}
\end{proposition}

\begin{proof}
    We are interested in a $\mathcal{H}$-representation of $\clPcone$. 
    Let us represent $U$ in the sphere, Figure \ref{figure:UinSphere}. 
 \begin{figure}
     \centering
     \begin{tikzpicture}
     \begin{scope}

        \filldraw[color=blue!60, fill = blue!5, very thick](0,0) circle (1);
        \draw[->,very thick](0,0) -- (0,1) node[above]{$u_1$};
        \draw[->,very thick](0,0) -- (-0.94,0.31) node[left]{$u_2$};
        \draw[->,very thick](0,0) -- (0,-1) node[below]{$u_3$};
        \draw[->,very thick](0,0) -- (0.98,0.16) node[right]{$u_4$};
        \draw [red,thick,domain=180:360] plot ({cos(\x)}, {sin(\x)});
        \end{scope}

       \begin{scope}[xshift=4cm]
       \draw[black] (1.18,-1.0) -- (0.85,1.0) -- (-0.73,1.0) -- (-1.39,-1.0) -- cycle;
       \filldraw[color = orange!60, fill = orange!30] (1/2,-1/4) -- (0.85,1.0) -- (0.06,1) node[black, above]{$\gamma_1$} -- (-0.73,1.0)   -- cycle;
       \filldraw[color = purple!60, fill = purple!30] (1/2,-1/4) -- (0.85,1.0) -- (1,0) node[black, right]{$\gamma_4$} --  (1.18,-1.0)  -- cycle;
       \filldraw[color = red!60, fill = red!30] (1/2,-1/4) -- (1.18,-1.0) -- (0.1,-1) node[black, below]{$\gamma_3$} -- (-1.39,-1.0)  -- cycle;
       \filldraw[color = blue!60, fill = blue!30] (1/2,-1/4) -- (-1.39,-1.0) -- (-1.05,0) node[black, left]{$\gamma_2$} --   (-0.73,1.0) -- cycle;
           
       \end{scope} 
       \begin{scope}[xshift=8cm]
       
       \draw[black, scale = 0.5] (1.18,-2.0) node[right]{$v_1$}  -- (-1.39,-2.0) node[left]{$v_2$} -- (0.33,4.22) node[above]{$v_3$} -- cycle;
       \filldraw[color = blue!60, fill = blue!30, scale = 0.5] (-1/2,-1) -- (-1.39,-2.0) -- (-0.503,1.11) node[black, left]{$\gamma_2$} -- (0.33,4.22)  -- cycle;
       \filldraw[color = purple!60, fill = purple!30, scale = 0.5] (-1/2,-1) -- (0.33,4.22) -- (0.75,1.11) node[black, right]{$\gamma_4$} -- (1.18,-2)  -- cycle;
       \filldraw[color = red!60, fill = red!30, scale = 0.5] (-1/2,-1) -- (1.18,-2) -- (0.1,-2) node[black, below]{$\gamma_3$} -- (-1.39,-2.0)  -- cycle;
           
       \end{scope} 
    \end{tikzpicture}
     \caption{The set $U$ drawn in the sphere with a open hemisphere $\omega \subset \mathbb{S}^1$  colored in red. The set $U$ allows us to construct trapezoids and triangles with the outer normals in $U$. Each $\gamma_i$ represents the cone-volume corresponding to the facet defined by $u_i$.}
     \label{figure:UinSphere}
 \end{figure}
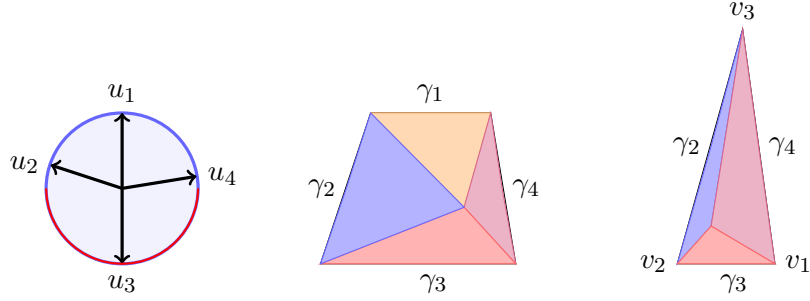
 From the setting in Theorem \ref{Theorem:ConeVolumeSetTrapezoids} we see that the set $\{ u_2,u_3,u_4 \}$ is a simplicial positive basis of $\R^2$. Therefore, there are polygons $P_2 = P(U,b_2),P_3=P(U,b_3),P_4=P(U,b_4)$ with the outer normal set $\{ u_2,u_3,u_4 \}$ and cone-volume vectors $e_2,e_3$ and $e_4$, respectively. For example, for the vertex $e_2$,  take the triangle from Figure \ref{figure:UinSphere} and shift it in such a way that the vertex $v_1$ is the origin and rescale it such that the volume is equal to $1$.
 Thus, the inequalities for the cone-volumes corresponding to $u_2,u_3$ and $u_4$ are $H^-(e_2,1)$, $H^-(e_3,1)$ and $H^-(e_4,1)$ and each inequality is also attained. In order to characterize $\clPcone$, we further need the inequalities for the cone-volume corresponding to the vector $u_1$. If we look at Figure \ref{FigurePconeTrapezoid} we compute that the inequalities are precisely 
 \begin{equation}
     H^-\left(e_1 + \frac{1}{2} e_3,\frac{1}{2}\right), \quad H^-(e_1 + e_2,1) \quad \text{ and } \quad H^-(e_1 + e_4,1).
 \end{equation}
 These inequalities are only obtained in the limit, as we will see later in Corollary \ref{Corollary:ConvergenceOfTrapezoidsAgaindsParallelotoptes}.
 We are able to represent $\clPcone$ as an intersection of halfspaces:
 \begin{align}
     \clPcone = & \bigcap_{i = 2}^4 H^-(e_i,1) \cap  H^-\left(e_1 + \frac{1}{2} e_3,\frac{1}{2}\right) \cap H^-(e_1 + e_2,1) \\ & \cap H^-(e_1 + e_4,1) \cap H(e_1 + e_2 + e_3 + e_4,1) \cap \bigcap_{i = 1}^4 H^+(e_i,0).
 \end{align}
 With the $\mathcal{H}$-representation established, we can compute the $\mathcal{V}$-representation:
 \begin{equation}
     \clPcone = \conv\left\{e_2,e_3,e_4,\frac{1}{2}(e_1 + e_2),\frac{1}{2}(e_1 + e_4)\right\}.
 \end{equation}
\end{proof}

We can see the vertices of $\clPcone$ in Figure \ref{FigurePconeTrapezoid}. The vertices correspond to the points $(0,1,0),(0,0,1),(0,0,0),(\frac{1}{2},\frac{1}{2},0)$ and $(\frac{1}{2},0,0)$, respectively. The $\mathcal{H}$-representation provided in Proposition \ref{Prop:VRepresentationTrapezoid} of $\clPcone$ is redundant: it consists of eleven inequalities, whereas $\clPcone$ only has five facets; see Figure \ref{FigurePconeTrapezoid}. 
Nevertheless, since $\dim(C_{\cv}(U)) =m-1 =3$ \cite[Proposition~2.7]{baumbach2025polynomialinequalitiesconevolumespolytopes}, and since both the vertex set and an $\mathcal{H}$-representation are available, a non-redundant $\mathcal{H}$-representation can be obtained by determining which inequalities support an $(m-2)$-dimensional face of $\clPcone$, that is, which inequalities actually define facets with respect to the given vertices. 
This approach, together with Theorem \ref{thm:VRepresentationOfclPconePlanar} and Theorem \ref{thm:HRepresentationOfclPconePlanar}, applies to any $U \in \mathcal{U}(2,m)$.

Notice that in the $\mathcal{V}$-representation of $\clPcone$, we have the vertices $e_i$ if and only if $u_i$ defines a triangle with two other vectors contained in $U$. This will also be made rigorous later. Further, the vertices $\frac{1}{2}(e_i + e_j)$ correspond to trapezoids such that $- u_i \in U$ and $u_j$ are the outer normals that define facets which meet the facet defined by $u_i$.
That is no coincidence, as we will later see. 

\vspace{\baselineskip}

Lastly, we briefly discuss the the situation $u_2 = - u_4$, that is, we have a parallelotope.

\begin{proposition}
\label{Prop:VRepresentationParallelogram}
    Let $U = \{ \pm u, \pm v \}$, where $v,u \in \mathbb{S}^1, u \not = v$. Then it holds 
    \begin{equation}
     \clPcone = \conv\left\{\frac{1}{2}(e_3 + e_2),\frac{1}{2}(e_3 + e_4),\frac{1}{2}(e_1 + e_2),\frac{1}{2}(e_1 + e_4)\right\}.
 \end{equation}
 as well as 
 \begin{align}
       \clPcone  & =  \\  \bigcap_{i = 1,3, j = 2,4} H^-\left(\frac{1}{2}(e_i + e_j),1\right) \cap H(e_1 + e_2 & + e_3 + e_4,1) \cap \bigcap_{i = 1}^4 H^+(e_i,0).
 \end{align}
\end{proposition}

\begin{proof}
    We have the following representation of $U$, Figure \ref{figure:UinSphere2}.
  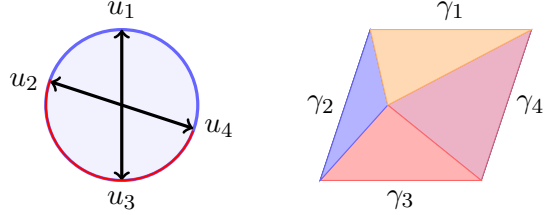
\begin{figure}[h]
     \centering
     \begin{tikzpicture}
     \begin{scope}

        \filldraw[color=blue!60, fill = blue!5, very thick](0,0) circle (1);
        \draw[->,very thick](0,0) -- (0,1) node[above]{$u_1$};
        \draw[->,very thick](0,0) -- (-0.94,0.31) node[left]{$u_2$};
        \draw[->,very thick](0,0) -- (0,-1) node[below]{$u_3$};
        \draw[->,very thick](0,0) -- (0.94,-0.31) node[right]{$u_4$};
        \draw [red,thick,domain=160:340] plot ({cos(\x)}, {sin(\x)});
        \end{scope}

       \begin{scope}[xshift=4cm]
       \draw[black] (0.73,-1.0) -- (1.39,1.0) -- (-0.73,1.0) -- (-1.39,-1.0) -- cycle;
       \filldraw[color = purple!60, fill = purple!30] (-1/2,0) -- (0.73,-1.0) -- (1.06,0) node[black, right]{$\gamma_4$} -- (1.39,1.0)   -- cycle;
       \filldraw[color = red!60, fill = red!30] (-1/2,0) -- (0.73,-1.0) --  (-0.33,-1) node[black, below]{$\gamma_3$} -- (-1.39,-1.0)  -- cycle;
       \filldraw[color = blue!60, fill = blue!30] (-1/2,0) --  (-0.73,1.0) -- (-1.06,0) node[black, left]{$\gamma_2$}  -- (-1.39,-1.0)  -- cycle;
       \filldraw[color = orange!60, fill = orange!30] (-1/2,0) -- (-0.73,1.0) -- (0.33,1) node[black, above]{$\gamma_1$} -- (1.39,1.0)  -- cycle;
           
       \end{scope} 
     
    \end{tikzpicture}
     \caption{The set $U$ with the open hemisphere $\omega \subset \mathbb{S}^1$ colored in red. This time we can only construct the parallelotope with the outer normals in $U$. Each $\gamma_i$ represents the cone-volume corresponding to the facet defined by $u_i$.}
     \label{figure:UinSphere2}
 \end{figure}

 In this case, the set $U$ is reducible as $\R^2 = \lin(\pm u_1) \oplus \lin(\pm u_2)$. From Proposition \ref{PropositionSeparatorsPsccAndConeVolumesDimension} it follows $P_\scc(U) = \Pcone$. 
 Hence, we get the following $\mathcal{V}$-representation 
 \begin{equation}
     \clPcone = \conv\left\{\frac{1}{2}(e_3 + e_2),\frac{1}{2}(e_3 + e_4),\frac{1}{2}(e_1 + e_2),\frac{1}{2}(e_1 + e_4)\right\}.
 \end{equation}
 as well as the $\mathcal{H}$-representation 
 \begin{align}
       \clPcone  & =  \\  \bigcap_{i = 1,3, j = 2,4} H^-\left(\frac{1}{2}(e_i + e_j),1\right) \cap H(e_1 + e_2 & + e_3 + e_4,1) \cap \bigcap_{i = 1}^4 H^+(e_i,0).
 \end{align}
\end{proof}

As we will later see, the three different shapes formed in the previous discussion, i.e. triangles, trapezoids and parallelograms, are the 'extreme cases', meaning that these cases correspond to vertices of the convex hull of the cone-volume set $\Pcone$ and these three different shapes characterize all the vertices of $\clPcone$.
To simplify terminology, a trapezoid is any quadrilateral with two parallel edges—in particular, parallelograms are trapezoids.

Since we want to be able to distinguish between vectors in $U$ that can form a triangle
and vectors in $U$ that do not have this property, we study the sets $U_\triangle$ and $U_\square$, as defined in Definition \ref{Definition:OuterNormalsInUsquareAdjacentFacets}.
The subindex of $\Usquare$ suggests that the elements contained in $\Usquare$ form a quadrilateral with additional properties. In fact, if $u \in \Usquare$ then $-u \in U$ and $u$ can define trapezoids with the outer normals contained in $U$. In addition, we will see that $|\Usquare| = 4$ if and only if $U$ defines a parallelogram. For these statements, we need an auxiliary lemma, which we state in a more general framework. 

\begin{lemma}
\label{Lemma:ExtendLinearBasisToPositiveSpanningSet}
    Let $U \in \mathcal{U}(2,m)$ and let $U' \subset U$ be a linear basis of $\R^n$. Then, there are linearly independent vectors $x_1,\dots,x_k$ contained in $U \setminus U', k \leq n$, such that $\pos(U' \cup \{ x_1,\dots,x_k \}) = \R^n$ and $|U' \cup \{ x_1,\dots,x_k \}| \leq 2n$.
\end{lemma}

\begin{proof}
    That follows directly from Carathéodory Theorem. 
\end{proof}

Lemma \ref{Lemma:ExtendLinearBasisToPositiveSpanningSet} implies that any vector in $U$ is the outer normal vector of a polygon with at most four facets. That observation is important for the structure of vectors contained in $\Usquare$. In fact, we show that all vectors contained in \(\Usquare\) correspond to trapezoids.

\begin{lemma}
    \label{Lemma:StructureUSquare}
    Let $U \in \mathcal{U}(2,m)$ and consider $u \in \Usquare$. 
    The following is true:
    \begin{enumerate}[label = (\roman*)]
        \item For every subset $\{x,y,z \} \subseteq U \setminus \{ u\}$ that fulfills $\pos\{u,x,y,z\} = \R^2$ it holds $-u \in \{ x,y,z \}$.
        \item There exists an open hemisphere $ \omega \subset \mathbb{S}^1$ such that $\{ -u \} = U \cap \omega$.
    \end{enumerate}  
    In particular, every subset $\{u,x,y,z \} \subseteq U$ that fulfills $\pos \{u,x,y,z \} = \R^2$ defines a trapezoid and $-u \in U$.
\end{lemma}

\textit{Proof.}
Let $u \in \Usquare$. There are vectors $x,y,z \in U \setminus \{ u \}$ such that $\pos\{u,x,y,z\} = \R^2$, Lemma \ref{Lemma:ExtendLinearBasisToPositiveSpanningSet}. Since $u \not \in \Utriangle$ for every proper subset $V \subsetneq \{x,y,z \}$ we have $\pos(\{u\} \cup V) \not = \R^2$. We wish to prove i), that is $-u \in \{x,y,z \}$.

First, w.l.o.g. we can assume $u = e_1$ and $x = e_2$. Otherwise, apply a linear isomorphism $\phi: \R^2 \rightarrow \R^2$, which maps $u$ to $e_1$ and $x$ to $e_2$. The positively spanning property $\pos\{e_1,e_2,y,z\} = \R^2$ implies the existence of scalars $\mu_y,\mu_z > 0$ such that 
\begin{equation}
    v \coloneqq  \mu_y y + \mu_z z < 0.
\end{equation}
Further, we can assume that $y_1 < 0$. If $y_2 = 0$, we have $y = -e_1$ and we are done. 

Therefore, assume that $y_2 \not = 0$. In that case $y_2 > 0$, since otherwise $\pos\{e_1,e_2,y\} = \R^2$. Hence, $z_2 < 0$. If in addition it holds $z_1 = 0$ we would get $\pos\{e_1, y,z\} = \R^2$, which is a contradiction to $u \in \Usquare = U \setminus \Utriangle$. We conclude that $z_1 > 0$ as before. But also in this case it would be true that $\pos\{e_1, y,z\} = \R^2$, due to the following observation 
\begin{align}
   e_2 & =  \frac{z_1}{-z_2} e_1 + \frac{1}{z_2} z \in \pos\{e_1,y,z\} \\
   -e_2 & = \frac{v_1}{v_2} e_1 + \frac{-1}{v_2} v \in \pos\{e_1,y,z\} \\
   - e_1 & = \frac{v_2}{v_1} e_2 + \frac{-1}{v_1} v \in \pos\{e_1,y,z\}.
\end{align}
The case $y_2 \not = 0$ always leads to a contradiction. Thus, $y_2 = 0$ must hold and we conclude $y = - e_1$.  

We have proven that for every $u \in \Usquare$ it holds $-u \in U$ and for every subset $\{x,y,z \} \subset U \setminus \{ u\}$ that fulfills $\pos\{u,x,y,z\} = \R^2$ the set $ \{ x,y,z \}$ contains the vector $-u$. 

We now prove the assertion about the hemisphere. Again, let $u \in \Usquare$. As before, we can assume that $u = e_1$, by applying a rotation $\pi : \R^2 \rightarrow \R^2$, which maps $u$ to $e_1$. Consider an element $z \in U$ such that 
\begin{equation}
    z_1 = \min_{v \in U \setminus \{ -e_1 \}} v_1 > - 1.
\end{equation}
It holds $z_1 \not = -1$. If $z_1 \geq 0$, we can choose $\omega \subset \mathbb{S}^1$ to be the open hemisphere $H^-(e_1,0) \setminus H(e_1,0)$ and we would be done.  Therefore, assume that $z_1 < 0$. Since $z \not = -e_1$ it either holds $z_2 > 0$ or $z_2 < 0$. We will only focus on the latter case, but the other case can be treated similarly.


Define the unit vector $ z^{\perp} \coloneqq (- z_2,z_1) $. We will prove that it must hold $\langle y,z^{\perp} \rangle \geq 0$ for every $y \in U \setminus \{ -e_1 \}$ which in turns would allow us to define the open hemisphere $\omega$ as $H^-(z^{\perp},0) \setminus H(z^{\perp},0)$.

First, assume that $y_2 \leq 0$. Due to the choice of $z$ it holds $y_1 \geq z_1$ and therefore also $y_2 \leq z_2$ which is equivalent to $(y_2 - z_2) \leq 0$. We compute 
\begin{equation}
    \langle y, z^{\perp} \rangle = (-z_2) y_1 + z_1 + y_2 \geq -z_2 z_1 + z_1 y_2 = z_1 (y_2 - z_2) \geq 0.
\end{equation}

Consider now the case $y_2 > 0$, but in addition it holds $\langle y,z^{\perp} \rangle < 0$. 
It must also hold $y_1 > 0$. For if $y_1 = 0$, it holds $\pos\{e_1,y,z\} = \R^2$, since we assume $z < 0$. The case $y_1 < 0$ also implies $\pos\{e_1,y,z\} = \R^2$, due to the following observation:
\begin{align}
   e_2 & =  \frac{-y_1}{y_2} e_1 + \frac{1}{y_2} y \in \pos\{e_1,y,z\} \\
   -e_2 & = \frac{z_1}{z_2} e_1 + \frac{-1}{z_2} z \in \pos\{e_1,y,z\} \\
   - e_1 & = \frac{z_2}{z_1} e_2 + \frac{-1}{z_1} z \in \pos\{e_1,y,z\}.
\end{align}
Hence, we have $z < 0$ and $y > 0$ and $u = e_1$. As in the calculations above, we conclude $-e_2 \in \pos\{e_1,y,z\}$. We wish to prove that also $-e_1 \in \pos\{e_1,y,z\}$. For that we consider the following intersection of the two lines $\conv\{y,z\}$ and $\pos\{-e_1\}$ and prove
\begin{equation}
    |\conv\{y,z\} \cap \pos\{-e_1\}| = 1.
\end{equation}
Since these are two different lines, the intersection can at most contain one element. In addition, it holds $z \in H(z^{\perp},0)$, $y \in H^-(z^{\perp},0) \setminus H(z^{\perp},0)$ and $e_1 \in H^+(z^{\perp},0 ) \setminus H(z^{\perp},0)$, as well as $z_1 <0$ and $y_1 > 0$. Thus, there is a scalar $\lambda > 0$ such that $(1 - \lambda) y + \lambda z \in \lin\{e_1\}$. Due to the halfspace relations it must hold $(1 - \lambda) y + \lambda z \in \pos\{-e_1\}$. This implies $-e_1 \in \pos\{e_1,y,z\}$ and also 
$e_2 = \frac{y_1}{y_2} \cdot (- e_1) + \frac{1}{y_2} y \in \pos\{e_1,y,z\}$. 

This contradiction implies $\langle y,z^{\perp} \rangle \geq 0$ and we can choose $\omega = H^-(z^{\perp},0) \setminus H(z^{\perp},0)$. 

The fact about the trapezoids is now an easy consequence of ii).
\qed
\newline 

  \begin{figure}[h]
     \centering
     \begin{tikzpicture}

        \filldraw[color=blue!60, fill = blue!5, very thick](0,0) circle (1);
        \draw[->,very thick](0,0) -- (0,1) node[above]{$e_1$};
        \draw[->,very thick](0,0) -- (-0.77,0.63) node[left]{$z^{\perp}$};
        \draw[->,very thick](0,0) -- (0.31,0.94) node[right]{$y$};
        \draw[->,very thick](0,0) -- (0,-1) node[below]{$-e_1$};
        \draw[->,very thick](0,0) -- (-0.77,-0.63) node[left]{$z$};
        \draw[->,very thick](0,0) -- (0.94,-0.31) node[right]{$y'$};
        \draw(-2,0) -- (2,0) node[right]{$H(e_1,0)$};
        \draw(-1.54,-1.26) -- (1.54,1.26) node[right]{$H(z^{\perp},0)$};
        \draw [red,thick,domain=220:398] plot ({cos(\x)}, {sin(\x)});
     
    \end{tikzpicture}
            \caption{Computing the hemisphere in the setting of Lemma \ref{Lemma:StructureUSquare}. The hemisphere $\omega \subset \mathbb{S}^1$ is colored red, and we rotated the set $U$ so that $u = e_1$. From the picture it is evident that 
            $\pos\{e_1,z,y'\} = \R^2$. }
     
    \end{figure}
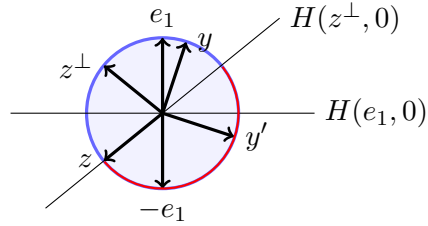

\begin{corollary}
\label{Corollary:StructureUSquare}
    For a set $U \in \mathcal{U}(2,m)$ it holds $|\Usquare| \leq 4$. Furthermore, equality occurs if and only if $U$ defines a parallelogram. 
\end{corollary}

\textit{Proof.} If $\Usquare = \emptyset$ there is nothing to prove. Therefore, assume $\Usquare \not = \emptyset$ and consider a vector $u \in \Usquare$. From Lemma \ref{Lemma:StructureUSquare} ii) we know that there exists an open hemisphere $\omega \subset \mathbb{S}^1$ such that $\omega \cap U = \{ -u \}$. We wish to prove that $\Tilde{U} \coloneqq  - \omega \cap ( \Usquare \setminus \{ u \}) = \emptyset$. Assume the contrary is true and consider a vector $v \in \Tilde{U}$. Since $v \in \Usquare$, it follows that $-v \in U \setminus \{ -u \} $ and also $-v \in \omega$. Therefore, we conclude that $-v \in (\omega \cap U) \setminus \{ -u \} $. But the last set is empty, a contradiction. Thus, 
\begin{equation}
    \Usquare \subseteq \{ \pm u \} \cup \left( \mathrm{cl}(\omega)\setminus \omega\right).
\end{equation}

Now, if $|\Usquare| = 4$, that is  $ \Usquare = \{ \pm u \} \cup \left( \mathrm{cl}(\omega)\setminus \omega\right)$,  it must hold \linebreak $ U = \{u_1,u_2,u_3,u_4 \} = \{ \pm u \} \cup \left( \mathrm{cl}(\omega) \setminus \omega \right)$ due to the following reasoning. Write $\omega_i$ for the open hemisphere contained in $\mathbb{S}^1$, such that $\{-u_i\} = U \cap \omega_i$. It is clear that $ \mathbb{S}^1 = \omega_1 \cup \omega_2 \cup \omega_3 \cup \omega_4 $, and thus 
\begin{align}
    U & = (\omega_1 \cap U) \cup (\omega_2 \cap U) \cup (\omega_3 \cap U) \cup (\omega_4 \cap U) \\ & =  \{-u_1\} \cup \{ -u_2\} \cup \{-u_3 \} \cup \{-u_4  \} = \{u_1,u_2,u_3,u_4  \}.
\end{align}
Hence, $U = \{ \pm u,\pm x \}$ for some $x \in \mathbb{S}^1$.

If $U$ defines a parallelogram, that is $U = \{ \pm u_1, \pm u_2 \}$ for some vectors $u_1,u_2 \in \mathbb{S}^1$, it is clear that $U = \Usquare$. \qed
\newline 

\begin{remark}
\label{Remark:StrauctureUSquare}
    In the setting of Corollary \ref{Corollary:StructureUSquare} it holds $|\Usquare| \in \{ 0,1,2,4\}$. The only interesting observation is that $|\Usquare| \not = 3$. Assume that $3 \leq |\Usquare|$. It follows that the set $\Usquare$ contains at least two antipodal points $\pm u$ (cf. proof of Corollary \ref{Corollary:StructureUSquare}). Thus, $U = \{ \pm u \} \cup \left( \mathrm{cl}(\omega)\setminus \omega\right)$ for some open hemisphere $\omega$ and we conclude that $\Usquare = U$ defines a parallelotope. 
\end{remark}

Lemma \ref{Lemma:StructureUSquare} states that every polygon with outer normals contained in $U$ and minimal number of facets is either a triangle or a trapezoid. The cone-volume vectors of triangles are easily understood, as triangles are just $2$-dimensional simplices. 
We also want to understand trapezoids better; to be more precise, we will focus on the limits of cone-volume vectors corresponding to trapezoids. This limit process is important for the study of the extreme points of the closure of the convex hull of $C_\cv(U)$.

\begin{corollary}
    \label{Corollary:ConvergenceOfTrapezoidsAgaindsParallelotoptes} 
    Let $U = \{ u_1,u_2,u_3,u_4 \} \subseteq \mathbb{S}^1$. Further, we assume that $u_1 = -u_3$ and that there is an open hemisphere $u_3 \in \omega \subset \mathbb{S}^1$ such that $u_2,u_4 \not \in \omega$.
    For the vertices $v_1 = (\frac{1}{2},\frac{1}{2},0,0), v_2 = (\frac{1}{2},0,0,\frac{1}{2}), v_3 = (0,\frac{1}{2},0,\frac{1}{2}), \linebreak  v_4 = (0,0,\frac{1}{2},\frac{1}{2}) \in P_\scc(U)$ there exist  sequences of right-hand sides $b_{i,n} \in \R_{\geq 0}^4$ such that the sequence of the corresponding cone-volume vectors converges to $v_i$ for $i = 1,2,3,4 \colon  $
    \begin{equation}
        \lim_{n \rightarrow \infty} \gamma(U,b_{i,n}) = v_i, \quad \text{ for } \quad i = 1,2,3,4
    \end{equation}
    and 
    \begin{equation}
        \gamma(U,b_{i,n}) \in P_\scc(U) \quad \text{ for all } \quad  n \in \N, i = 1,2,3,4.
    \end{equation}    
\end{corollary}

\textit{Proof.} First, notice that the sets of vectors $\{u_1,u_2 \}$, $\{u_1,u_4\}$, $\{u_2,u_3\}$ and $\{u_3,u_4\}$ are linearly independent, respectively. Thus, the vectors $v_1,v_2,v_3$ and $ v_4 $ are contained in $  P_\scc(U)$.  
Since it holds $P_\scc(U) \subseteq \clPcone$ (Proposition \ref{Prop:VRepresentationTrapezoid} and  \ref{Prop:VRepresentationParallelogram}), the statement follows.\qed
\newline

Assume the setting of Corollary \ref{Corollary:ConvergenceOfTrapezoidsAgaindsParallelotoptes}.  
We note that Corollary \ref{Corollary:ConvergenceOfTrapezoidsAgaindsParallelotoptes} does not explicitly address, for example, the vertex 
\((0,\tfrac{1}{2},\tfrac{1}{2},0) \in P_\scc(U)\).  
However, in the case of a trapezoid, either the set forms a parallelogram or the vectors 
\(\{u_2, u_3, u_4\}\) span a triangle.  
In both cases, the vertex \((0,\tfrac{1}{2},\tfrac{1}{2},0)\) is contained in the convex hull of \(\clPcone\); cf. Propositions \ref{Prop:VRepresentationTrapezoid} and \ref{Prop:VRepresentationParallelogram}.

Using this observation together with Corollary \ref{Corollary:ConvergenceOfTrapezoidsAgaindsParallelotoptes} we provide a different proof for the fact that the subspace concentration polytope is (at least) contained in the convex hull of the cone-volume set, in the planar case (cf. Theorem \ref{StancuSufficientConditions} iii)). 

\begin{proposition}
    \label{Proposition:PlanarPsccSubsetConvPcone}
    Let $U \in \mathcal{U}(2,m)$. Then it holds 
    \begin{equation}
        P_\scc(U) \subseteq \clPcone.
    \end{equation}
\end{proposition}

\textit{Proof.} Assume first that $U$ is reducible, that is, there exists a subset $V \subsetneq U$ such that $\R^2 = \lin(V) \oplus \lin(U \setminus V) $. In that case we are working in $1$-dimensional spaces and there the subset relation is clear. 

Therefore, assume that $U$ is irreducible. 
We have to prove that every vertex of $\Pscc$ is contained in $\clPcone$. Let $v \in \Pscc$ be a vertex corresponding to the linear basis $u_1,u_2 \in U$, that is, $v = \frac{1}{2}(e_1 + e_2)$. We consider two different cases. 

First, assume that $u_1,u_2 \in \Utriangle$. In that case the cone-volume set $\Pcone$ contains the vectors $e_1$ and $e_2$ and thus $v \in \clPcone$. 

Suppose now that $u_1 \in \Usquare$ or $u_2 \in \Usquare$. We only consider the case $u_1 \in \Usquare$, the other case can be treated similarly. There are two vectors $y_1,y_2 \in U$ such that $\pos(u_1,u_2,y_1,y_2) = \R^2$ and the set $\{ u_1,u_2,y_1,y_2 \}$ is the outer normal set of a trapezoid, Lemma \ref{Lemma:StructureUSquare}. Corollary \ref{Corollary:ConvergenceOfTrapezoidsAgaindsParallelotoptes} implies the existence of a sequence of cone-volume vectors contained in $\Pcone$ that converges to $v$. Hence, also in this case $v \in \clPcone$. 
\qed
\newline

Our next goal is to provide a $\mathcal{V}$-representation of $\clPcone$. For that, we recall the set that captures the adjacent outer normals of elements contained in $\Usquare$. Consider a set $U \in \mathcal{U}(2,m)$. For an outer normal vector $u \in \Usquare$ we defined the set
\begin{equation}
    U_{\square,u} = \{ x,y \in U :  \pos\{u,-u,x,y\} = \R^2 \}.
\end{equation}
The set $U_{\square,u}$, as defined in Definition \ref{Definition:OuterNormalsInUsquareAdjacentFacets}, consists precisely of the outer normals in $U$ that correspond to facets of trapezoids with outer normal vectors $\pm u$, which are adjacent to the facet defined by $u$.

As mentioned after Proposition \ref{Prop:VRepresentationParallelogram}, the $\mathcal{V}$-representation of $\clPcone$ depends on the three sets $\Utriangle, \Usquare$ and $U_{\square,u}$.
We recall Theorem \ref{thm:VRepresentationOfclPconePlanar}.

\VRepresentation*

\bigskip

In particular, equality holds if and only if $U$ defines a parallelogram.

Let us have a look at an example before we prove the main Theorem  \ref{thm:VRepresentationOfclPconePlanar}.

\begin{example}
      Let $U = \{ u_1,u_2,u_3,u_4 \} \subset \mathbb{S}^1$ define a trapezoid, that is, we assume that $u_1 = -u_3$ and there is an open hemisphere $u_3 \in \omega \subset \mathbb{S}^1$ such that $u_2,u_4 \not \in \omega$.

      In the Propositions \ref{Prop:VRepresentationTrapezoid} and \ref{Prop:VRepresentationParallelogram} we have seen that there are two different cases, namely the set $U$ defines either trapezoids and triangles or parallelograms. 

      In the trapezoid case, the vectors $u_2,u_3,u_4$ are contained in the set $\Utriangle$ and $u_4 \in \Usquare$. In addition, we have $U_{\square,u_1} = \{ u_2,u_4 \}$. Theorem  \ref{thm:VRepresentationOfclPconePlanar} tells us that 
      \begin{equation}
          \clPcone = \conv\left\{e_2,e_3,e_4,\frac{1}{2}(e_1 + e_2), \frac{1}{2}(e_1 + e_4)\right\}.
      \end{equation}
      That is consistent with the $\mathcal{V}$-representation, that we computed in Proposition \ref{Prop:VRepresentationTrapezoid}.

      In the latter case $u_1,u_2,u_3,u_4$ are all contained in $\Usquare$ and for example $U_{\square,u_1} = \{ u_2,u_4 \}$. Thus, Theorem \ref{thm:VRepresentationOfclPconePlanar} says that the $\mathcal{V}$-representation of $\clPcone$ is given by 
      \begin{equation}
          \clPcone = \conv\left\{\frac{1}{2}(e_1+e_2),\frac{1}{2}(e_1+e_4),\frac{1}{2}(e_3+e_2),\frac{1}{2}(e_3+e_2)\right\}.
      \end{equation}
\end{example}

\bigskip

The proof of Theorem \ref{thm:VRepresentationOfclPconePlanar} involves several steps.

First, we show that the cone-volume corresponding to a vector $u \in \Usquare$ is always less than or equal to $\frac{1}{2}$. That will be done by iteratively removing facets of a given polytope $P(U,b)$, $b \in \R_{\geq 0}^{m}$, $\mathrm{vol}(P(U,b)) = 1$, that meet the facet defined by $u$, say $F_u$, showing that this will only increase the cone-volume corresponding to $u$. We will do this until we have constructed a trapezoid with an outer normal $u$. It is important to note that at first glance it seems trivial that, by removing a facet that meets $F_u$, we can only increase the cone-volume corresponding to $u$. However, we have to rescale the new polygon in order to have volume equal to $1$. 

Second, we will provide a $\mathcal{H}$-representation, corresponding to the $\mathcal{V}$-representation stated in Theorem \ref{thm:VRepresentationOfclPconePlanar}, which contains $\clPcone$ and compute the vertices of that given $\mathcal{H}$-representation. 

Lastly, we show that the vertices of the $\mathcal{H}$-representation are all contained in $\mathrm{cl}(\Pcone)$ and this will conclude the proof of Theorem \ref{thm:HRepresentationOfclPconePlanar}.

We begin with a monotonicity statement, which implies that the cone-volume for vectors $u \in \Usquare$ is less than or equal to $\frac{1}{2}$. Let $U = \{ u_1,\dots,u_m \} \in \mathcal{U}(2,m)$ be a set of counterclockwise ordered unit vectors and fix a polygon $P = P(U,b)$, $b \in \R_{\geq 0}^{|U|}$ with volume equal to $1$ such that every unit vector in $U$ defines a facet of $P$. Further, assume that for a fixed index $i$, removing the facet defined by the outer normal $u_i$ results in a bounded polyhedron $P'$, i.e., $P'$ is a polygon. Define the polygon $Q \coloneqq \frac{1}{\sqrt{\mathrm{vol} (P')}} P'$. Since the set of outer unit normals of $Q$ is a subset of $U$, there exists a right-hand side $b' \geq 0$ such that $Q = P(U,b')$. If one of the adjacent outer unit normal vectors $u_{\mod(i-1)}$ or $u_{\mod(i+1)}$ is contained in $\Usquare$, the following inequality holds. For the next lemma we write $\gamma(U,b)_u$ for the entry of $\gamma(U,b)$ that corresponds to $u \in U$.

\begin{lemma}
\label{Lemma:DeltingAjacentFacetInceasreConeVolume}
    For the two polygons $P = P(U,b)$ and $Q = P(U,b')$ it holds:
    \begin{enumerate}[label = (\roman*)]
        \item If $u_{mod (i - 1)} \in \Usquare $, it holds $\gamma(U,b')_{u_{mod (i - 1)}} \geq \gamma(U,b)_{u_{mod (i-1)}}$.
        \item If $u_{mod (i +1)} \in \Usquare $, it holds $\gamma(U,b')_{u_{mod (i + 1)}} \geq \gamma(U,b)_{u_{mod (i+1)}}$.
    \end{enumerate}
   
\end{lemma}

\textit{Proof.} 
W.l.o.g. we assume that $i = 2$ and $u_1 = e_1 \in \Usquare$. Otherwise, use a rotation $\phi: \R^2 \rightarrow \R^2$, such that $\phi(u_1) = e_1$. Further, notice that there exists no open hemisphere $\omega \subset \mathbb{S}^1$, such that $u_2 \in \omega$, but $u_1$ and $u_3$ are not contained in $\omega$, since otherwise the polyhedron $P'$ is no longer bounded. 

We consider the facets $F_i(b) \subset P = P(U,b)$ for the facets defined by $u_i$, $i = 1,2,3$. We know $F_{1}(b) = \conv(c,a)$ and $F_2(b) = \conv(a,b)$ and $F_3(b) = \conv(b,d)$  for some points $a,b,c,d \in \R^2$. Further, define $ x= (x_1,x_2) \coloneqq  \aff(F_1(b)) \cap \aff (F_3(b))$ and the triangle $\Delta = \conv(a,b,x)$. It holds $P' = P \cup \Delta$ and $\mathrm{vol}(P') = \mathrm{vol}(P) + \mathrm{vol}(\Delta) = 1 + \mathrm{vol}(\Delta )$ and $x_1 = a_1$. We have the following volume formulas 
\begin{align}
    \gamma(U,b)_{u_1} & = \frac{(a_2 - c_2) \cdot a_1}{2} \\
    \gamma(U,b')_{u_1} & = \frac{(x_2 - c_2) \cdot a_1}{2 (1 + \mathrm{vol}(\Delta))} \\  
    \mathrm{vol}(\Delta) & = \frac{(x_2 - a_2) \cdot (a_1 - b_1)}{2}.
\end{align}
We will prove that $\gamma(U,b')_{u_1} \geq \gamma(U,b)_{u_1}$. If $a_1 = 0$ there is nothing to prove. Therefore, assume $a_1 > 0$. The inequality is equivalent to 
\begin{equation}
    \frac{(x_2 - c_2)}{(1 + \mathrm{vol}(\Delta))} \geq (a_2 - c_2),
\end{equation}
which can be transformed into 
\begin{equation}
    (x_2 - c_2) \geq (a_2 - c_2) + \mathrm{vol}(\Delta) \cdot (a_2 - c_2).
\end{equation}
Since $x_2 - c_2 = (x_2 - a_2) + (a_2 - c_2)$ it is enough to prove 
\begin{equation}
    (x_2 - a_2) \geq \mathrm{vol}(\Delta) \cdot (a_2 - c_2),
\end{equation}
or equivalently 
\begin{equation}
    2 \geq (a_1 - b_1) \cdot (a_2 - c_2).
\end{equation}
That is true due to the following observation. The parallelotope 
\begin{equation}
    A \coloneqq H^-(e_1,a_1) \cap H^-(-e_1,b_1) \cap H^-(u_{3},\langle a,u_3 \rangle) \cap H^-(-u_3, \langle c, u_3 \rangle)
\end{equation} is contained in the polytope $P$, Lemma \ref{Lemma:StructureUSquare} ii). Therefore, $\mathrm{vol}(A) \leq \mathrm{vol}(P) \leq 1$ and further, $\mathrm{vol}(A) = (a_1 - b_1) \cdot (a_2 - c_2)$. 
\qed

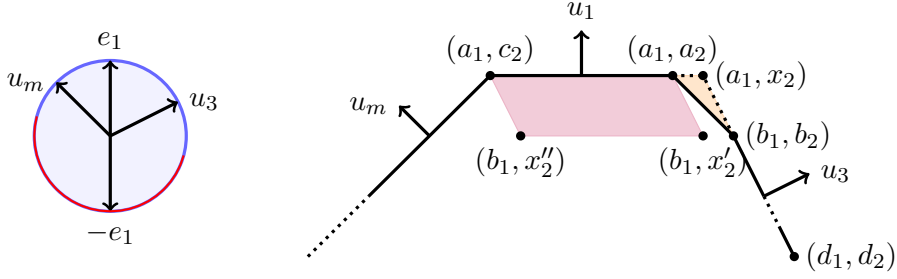
\begin{figure}
    \centering
        
        \begin{tikzpicture}[scale = 1]
        \begin{scope}

        \filldraw[color=blue!60, fill = blue!5, very thick](0,0) circle (1);
        \draw[->,very thick](0,0) -- (0,1) node[above]{$e_1$};
        \draw[->,very thick](0,0) -- (0.89,0.45) node[right]{$u_3$};
        \draw[->,very thick](0,0) -- (0,-1) node[below]{$-e_1$};
        \draw[->,very thick](0,0) -- (-0.71,0.71) node[left]{$u_m$};
        \draw [red,thick,domain=165:345] plot ({cos(\x)}, {sin(\x)});
     \end{scope}

    \begin{scope}[xshift=5cm,scale = 0.8]
    \filldraw[color = purple!30, fill = purple!20] (0.5,0) --  (0,1.0) -- (3,1) -- (3.5,0)  -- cycle;
        \filldraw[color = orange!30, fill = orange!20] (3,1) --  (4,0) -- (3.5,1) -- cycle;
        \filldraw[color=black!100, fill = black!100, very thick](0.5,0) circle (0.05) node[below]{$(b_1,x_2'')$};
        \filldraw[color=black!100, fill = black!100, very thick](3.5,0) circle (0.05) node[below]{$(b_1,x_2')$};
    
        \filldraw[color=black!100, fill = black!100, very thick](3,1) circle (0.05) node[above]{$(a_1,a_2)$};
        \filldraw[color=black!100, fill = black!100, very thick](4,0) circle (0.05) node[right]{$(b_1,b_2)$};
        \filldraw[color=black!100, fill = black!100, very thick](0,1) circle (0.05) node[above]{$(a_1,c_2)$};
        \filldraw[color=black!100, fill = black!100, very thick](3.5,1) circle (0.05) node[right]{$(a_1,x_2)$};
        \filldraw[color=black!100, fill = black!100, very thick](5,-2) circle (0.05) node[right]{$(d_1,d_2)$};
        \draw[->,very thick] (1.5,1) -- (1.5,1.75) node[above]{$u_{1}$};
        \draw[->,very thick] (4.5,-1) -- (5.25,-0.625) node[right]{$u_{3}$};
        \draw[-,very thick](-2,-1) -- (0,1);
        \draw[dotted,very thick](-3,-2) -- (-2,-1);
        \draw[-,very thick](0,1) -- (3,1);
        \draw[-,very thick](3,1) -- (4,0);
        \draw[-,very thick](4,0) -- (4.5,-1);
        \draw[dotted,very thick](3,1) -- (3.5,1);
        \draw[dotted,very thick](4,0) -- (3.5,1);
        \draw[dotted,very thick](4.5,-1) -- (4.75,-1.5);
        \draw[-,very thick](4.75,-1.5) -- (5,-2);
        \draw[->,very thick](-1,0) -- (-1.5,0.5) node[left]{$u_m$};

     \end{scope}
    \end{tikzpicture}
        
    \caption{The setting of Lemma \ref{Lemma:DeltingAjacentFacetInceasreConeVolume}, where we assume that $u_{i-1} = e_1$. On the left is a representation of $U$ drawn into the sphere. The hemisphere 'separating' $u_3$ and $u_m$ is colored in red. On the right we see a picture of the polytope $P$, which is limited by the black lines. The polytope $P'$ is the polytope $P$ togehter with the orange triangle. The red area corresponds to the parallelotope $A$ and the orange area is the triangle $\Delta$. }
    \label{fig:enter-label}
\end{figure}

\begin{corollary}
    \label{corollary:ConeVolumeUSquareLessOneHalf}
    Let $U \in \mathcal{U}(2,m)$ and $u \in \Usquare$. For any $\gamma(U,b)  \in C_\cv(U)$ it holds $\gamma(U,b)_u \leq \frac{1}{2}$.
\end{corollary}

\textit{Proof.} Take a right-hand side $b \in \R_{\geq 0}^m$ and consider the polygon $P = P(U,b), \mathrm{vol}(P)=1$. Assume that $u$ defines a facet of $P$, say $F_u$, otherwise the cone-volume with respect to $u$ would be zero.

If $P$ is a polygon with more than $4$ facets, we delete a facet that is adjacent to $F_u$, i.e. the two facets have a non-trivial intersection. The resulting polyhedron will still be a polygon (cf. Lemma \ref{Lemma:ExtendLinearBasisToPositiveSpanningSet}). Lemma $\ref{Lemma:DeltingAjacentFacetInceasreConeVolume}$ implies that if we can delete an adjacent facet of $F_u$ then the new cone-volume corresponding to $u$ can only increase. We repeat that process until the resulting polygon, say $Q$, has exactly $4$ facets and since $u \in \Usquare$, it is, in fact, a trapezoid. Now the cone-volume of $Q$ corresponding to $u$ is greater or equal than the cone-volume $\gamma(U,b)_u$, that we started with. But since we now deal with a trapezoid, we know the structure of the cone-volume vector of $Q$, in particular, $\gamma(U,b)_u \leq \frac{1}{2}$ (cf. Theorem \ref{Theorem:ConeVolumeSetTrapezoids} and Proposition \ref{Prop:VRepresentationTrapezoid}). 
\qed
\newline 

Next, we study the polytope $K_U$ for a given set $U \in \mathcal{U}(2,m)$, given by the following $\mathcal{H}$-representation:
\begin{align}
    K_U \coloneqq & \left( \bigcap_{i: u_i \in \Utriangle} H^-(e_i,1) \right) \cap \left( \bigcap_{i : u_i \in \Usquare, - u_j = u_i} H^- \left(e_i + \frac{1}{2} e_j, \frac{1}{2} \right) \right) \\
     & \cap \left( \bigcap_{i : u_i \in \Usquare, u_k \in U_{\square,u_i}} H^-(e_i + e_k , 1)  \right) \cap \left( \bigcap_{i = 1}^m H^-(-e_i,0) \right) \\ 
     & \cap H(e_1 + \dots + e_m, 1).
\end{align}

The polytope $K_U$ is, in fact, the convex hull $\clPcone$. We need the following auxiliary statements. 

\begin{lemma}
\label{Lemma:PconeSubsetKU}
    Let $U \in \mathcal{U}(2,m)$. Then it holds 
    \begin{equation}
        \clPcone \subseteq K_U.
    \end{equation}
\end{lemma}

\textit{Proof.}
It is enough to prove that $\Pcone \subseteq K_U$.
We assume that $U = ( u_1,\dots,u_m )$ is a set tuple of counterclockwise ordered unit vectors.
Consider a polytope $P = P(U,b)$, $b \in \R_{\geq 0}^m$, such that $\mathrm{vol}(P) = 1$. The only nontrivial inclusion is 
\begin{equation}
    \Pcone \subseteq \bigcap_{(i,j) : u_i \in \Usquare, - u_j = u_i} H^- \left(e_i + \frac{1}{2} e_j, \frac{1}{2} \right).
\end{equation}
We assume that $u_1 \in \Usquare$, and $u_1 = e_1$. And there is a $j > 2$ such that $u_j = -e_1$. We write $F_1 = \conv(a,b)$ for the face defined by $e_1$ and $F_2 = \conv(x,y)$ for the face defined by $u_j$. As in Lemma \ref{Lemma:DeltingAjacentFacetInceasreConeVolume}, we conclude that the parallelotope 
\begin{equation}
    A \coloneqq H^-(e_1,a_1) \cap H^-(-e_1,0) \cap H^-(u_{2},\langle a,u_2 \rangle) \cap H^-(-u_2, \langle b, u_2 \rangle)
\end{equation}
is contained in $P$. In addition, $\vol(A) = 2 \cdot \gamma(U,b)_1$. The simplex $\Delta \coloneqq \conv(0,x,y)$ fulfills $\mathrm{vol}(\Delta) = \gamma(U,b)_j$. Since $\mathrm{int}(A) \cap \mathrm{int}(\Delta) = \emptyset$ and also 
$\Delta \cup A \subset P$, we conclude 
\begin{equation}
    2 \cdot \gamma(U,b)_1 + \gamma(U,b)_j = \mathrm{vol}(A \cup \Delta) \leq \mathrm{vol}(P) = 1,
\end{equation}
which is equivalent to 
\begin{equation}
    \gamma(U,b)_1 + \frac{1}{2} \gamma(U,b)_j \leq \frac{1}{2}.
\end{equation}
\qed

Our next goal is to compute the $\mathcal{V}$-representation of $K_U$. 
We use the characterization from Barvinok \cite[Theorem 4.2]{barvinok2002course} stating that a point of a polyhedron is a vertex exactly when the normals of the active constraints at that point span the entire space.

\begin{lemma}
    \label{Lemma:VRepresentationOfKU}
    It holds 
    \begin{equation}
        K_U = \conv\left\{ \bigcup_{i : u_i \in \Utriangle} \{ e_i \} \cup \bigcup_{(i,j) : u_i \in \Usquare, u_j \in U_{\square,u_i} } \left\{ \frac{1}{2}(e_i + e_j) \right\} \right\}.
    \end{equation}
\end{lemma}

\textit{Proof.} 
The inclusion that the right-hand side is contained in $K_U$ is trivial. By a straightforward computation one can check that each element in the set on the right-hand side is contained in $K_U$.

We now show that $K_U$ is contained in the right-hand side, by computing the vertices of $K_U$ with \cite[Theorem 4.2]{barvinok2002course}.
We will consider two different cases. 

First, assume that $\Usquare$ contains two antipodal points, i.e. there is a $u \in \mathbb{S}^1$ such that $\pm u \in \Usquare$. In that case $U$ defines a parallelogram, Remark \ref{Remark:StrauctureUSquare} and we have already established the structure of parallelograms. 

We now assume that $\Usquare$ does not contain antipodal points. We want to compute the vertices of $K_U$ by using \cite[Theorem 4.2]{barvinok2002course}. Therefore, we consider different cases in which inequalities are satisfied. Let $v \in K_U$ be a vertex. 

Case 1: Assume that $v$ satisfies $\langle v, e_i \rangle = 1$ for an index $i$ such that $u_i \in \Utriangle$. Then it is clear that $v = e_i$.

Case 2: Assume that $\langle v, e_i \rangle < 1$ for all indices $i$ such that $u_i \in \Utriangle$. 

Case 2a): The vertex $v$ satisfies $\langle v, e_i + e_k \rangle = 1$ for some indices $i,k$ such that $u_i \in \Usquare$ and $u_k \in U_{\square,u_i}$. Since $v \in K_U$, it holds $v_j = 0$ for all $j \not = i,k$. Since 
\begin{equation}
    \lin\{e_i + e_k, e_1 + \dots + e_m, e_j : j \not = i,k\} \subsetneq \R^m,
\end{equation}
\cite[Theorem 4.2]{barvinok2002course} implies $\frac{1}{2} = \langle v , e_i + \frac{1}{2} e_j \rangle = v_i$ for some index $j \not = k$. Thus, $v = \frac{1}{2} e_i + \frac{1}{2} e_k$. 

Case 2b): The vertex $v$ satisfies $\langle v, e_i + e_k \rangle < 1$ for all indices $i,k$ such that $u_i \in \Usquare$ and $u_k \in U_{\square,u_i}$.
Thus there exist indices $i,j$ such that $u_i \in \Usquare$ and $u_j \in U_{\square,u_i}$ and $\langle v, e_i + \frac{1}{2} e_j \rangle = \frac{1}{2}$. It must hold $v_i > 0$, since we assume $v_k < 1$ for all $k = 1,\dots,m$.

From our assumption about $\Usquare$ it must hold $-u_i \not \in \Usquare$.
We prove that $v_j = 0$. Therefore, assume that $v_j > 0$. Again, \cite[Theorem 4.2]{barvinok2002course} implies the existence of two indices $k,l$ such that $k \neq i$ and $u_k \in \Usquare$ and $u_l \in U_{\square,u_k}$ and $\langle v, e_k + \frac{1}{2} e_l \rangle = \frac{1}{2}$ and w.l.o.g. we assume $v_k \not = 0$.
Since $-u_i \not \in \Usquare$, it holds $u_k \not = - u_i$. However, in this case we would have
\begin{equation}
    1 = v_i + \frac{1}{2} v_j + v_k + \frac{1}{2} v_l < v_i + v_j + v_k + v_l \leq 1. 
\end{equation}
Hence, $v_j = 0$ and $v_i = \frac{1}{2}$.
Since the set of active inequalities on $v$ linearly spans $\R^m$, \cite[Theorem 4.2]{barvinok2002course} again, we know there exist indices $h,g$ such that $u_h \in \Usquare$, $u_g \in U_{\square,u_h}$ and   $\langle v, e_h + \frac{1}{2} e_g \rangle = \frac{1}{2}$. Just as before, we conclude that $v_h = \frac{1}{2}$ and $v_g = 0$. We get $v = \frac{1}{2}(e_i + e_h)$. 

Whats left is to check that $u_h \in U_{\square,u_i}$. But this is clear, since $\{ \pm u_i,\pm u_h \} \in U$, Lemma \ref{Lemma:StructureUSquare}, and the set $\{ \pm u_i,\pm u_h \}$ defines a parallelogram. 

These are also all vertices of $K_U$. Since if, in addition to the cases before, we also assume that the vertex $v$ satisfies $\langle v, e_i + \frac{1}{2} e_j \rangle < \frac{1}{2}$ for all indices $i,j$ such that $u_i \in \Usquare$ and $u_j = -u_i$ we cannot generate new vertices. That is due to the fact that at least two entries of $v$ are not equal to zero, say $v_i,v_j \not = 0$, and thus
\begin{equation}
    \lin\{e_1 + \dots+ e_m , e_s : s \not = i,j \} \subsetneq \R^m.
\end{equation}

\qed

\bigskip
We now prove that the vertices of $K_U$ are contained in $\mathrm{cl}({\Pcone})$.

\begin{lemma}
    \label{Lemma:VRepresentationOfKUcontainedinPcone}
    Let $U \in \mathcal{U}(2,m)$. We have the following subset relation 
    \begin{equation}
        \bigcup_{i : u_i \in \Utriangle} \{ e_i \} \cup \bigcup_{(i,j) : u_i \in \Usquare, u_j \in U_{\square,u_i} } \left\{ \frac{1}{2}(e_i + e_j) \right\}  \subseteq \mathrm{cl}({\Pcone}).
    \end{equation}
\end{lemma}

\textit{Proof.} 
If $u_i \in \Utriangle$, it is straightforward, consider the case $\{u_2,u_3,u_4 \}$ in Proposition \ref{Prop:VRepresentationTrapezoid}. For indices $i,j$ such that $u_i \in \Usquare$ and $u_j \in U_{\square,u_i}$, Corollary \ref{Corollary:ConvergenceOfTrapezoidsAgaindsParallelotoptes} implies the existence of cone-volume vectors that converge to $\frac{1}{2}(e_i + e_j)$. \qed

\bigskip
We are ready to prove Theorem \ref{thm:VRepresentationOfclPconePlanar} and \ref{thm:HRepresentationOfclPconePlanar}. 
\newline 

\textit{Proof of Theorem \ref{thm:VRepresentationOfclPconePlanar} + \ref{thm:HRepresentationOfclPconePlanar}.}
Lemma \ref{Lemma:PconeSubsetKU} and Lemma \ref{Lemma:VRepresentationOfKU} together imply
\begin{align}
    & \conv\left\{ \bigcup_{i : u_i \in \Utriangle} \{ e_i \} \cup \bigcup_{(i,j) : u_i \in \Usquare, u_j \in U_{\square,u_i} } \left\{ \frac{1}{2}(e_i + e_j) \right\} \right\} \\
     \subseteq & \clPcone \\
     \subseteq & K_U \\
     = &\conv\left\{ \bigcup_{i : u_i \in \Utriangle} \{ e_i \} \cup \bigcup_{(i,j) : u_i \in \Usquare, u_j \in U_{\square,u_i} } \left\{ \frac{1}{2}(e_i + e_j) \right\} \right\}.
\end{align}
\qed 

We were able to provide a $\mathcal{V}$- and $\mathcal{H}$-representation of $\clPcone$. Further, the vertices of $\clPcone$ are precisely the elements that are contained in the $\mathcal{V}-$representation stated in Theorem  \ref{thm:VRepresentationOfclPconePlanar} 

\begin{corollary}
\label{Corollary:VerticesOfPconePlanar}
    The set of vertices of the polytope $\clPcone$ is equal to
    \begin{equation}
         \bigcup_{i : u_i \in \Utriangle} \{ e_i \} \cup \bigcup_{(i,j) : u_i \in \Usquare, u_j \in U_{\square,u_i} } \left\{ \frac{1}{2}(e_i + e_j) \right\}.
    \end{equation}
\end{corollary}

\textit{Proof.}
For the indices $i$ such that $u_i \in \Utriangle$, it is clear that $e_i$ is a vertex of $\clPcone$. Therefore, consider an index $i$ such that $u_i \in \Usquare$. Let $j$ be an index such that $u_j \in U_{\square,u_i}$. Corollary \ref{corollary:ConeVolumeUSquareLessOneHalf} implies that $\gamma(U,b)_{u_i} \leq \frac{1}{2}$ for any cone-volume vector $\gamma(U,b) \in \clPcone$. Suppose we have two elements $v,w \in \clPcone$ such that 
\begin{equation}
    \frac{v + w}{2} = \frac{1}{2} (e_i + e_j).
\end{equation}
Hence, it must hold $v_{u_i} = w_{u_i} = 1/2$, Corollary \ref{corollary:ConeVolumeUSquareLessOneHalf}. Since, in addition, $|v|_1 = |w|_1 = 1$, it must hold $v_{u_j} =  w_{u_j} = 1/2$ and we are done. 
\qed
\newline

It is remarkable that the extreme points of the set $\clPcone$ correspond to either triangles, the elements contained in $U_\triangle$, or the vertices of $P_\scc(U)$. This behaviour is no longer true for general dimensions.

\begin{example}
    Consider the set $$U = \{u_1,\dots,u_{n+2} \} = \left\{ - e_1, e_1,e_2,\dots,e_n,-\sqrt{n} \cdot (e_1 + \dots + e_n) \right\} \in \mathcal{U}(n,n+2).$$ Since $U$ contains the simplex $V = \{e_1,\dots,e_n,-\sqrt{n} \cdot (e_1 + \dots + e_n) \}$, we have $e_2,\dots,e_{n+2} \in \clPcone$.
    Consider a polytope $P = P(U,b), b \geq 0$ and assume that $-e_1$ defines a facet of $P$, say $F_1(b)$. The set \linebreak $\conv(F_1(b), F_1(b) + b_1 e_1)$ is a subset of $P$ and hence we get the following inequality for the cone-volume vectors $\gamma \in C_{\cv}(U)$,
    \begin{equation*}
        n \cdot \gamma_1 + \gamma_2 \leq 1.
    \end{equation*}
    Using this inequality we get the following $\mathcal{V}$-representation of $\clPcone$:
    \begin{align*}
        \clPcone & = \\ \conv\Bigg\{e_2,\dots,e_{n+2},  \left(\frac{1}{n},\frac{n-1}{n},0,\dots,0\right), & \dots,\left(\frac{1}{n},0,\dots,0,\frac{n-1}{n}\right)\Bigg\}.
    \end{align*}
    Notice that no vertex in this $\mathcal{V}$-representation is a vertex of $P_\scc(U)$ and not every vertex corresponds to a simplex either.
\end{example}
 
We now present an alternative proof of the fact that the cone-volume set coincides with the subspace concentration polytope if and only if the set $U$ defines a parallelogram (cf. \cite[Thm 2.11]{baumbach2025polynomialinequalitiesconevolumespolytopes}), using the $\mathcal{V}-$representation of $\clPcone$. 

\begin{corollary}
    Let $U \in \mathcal{U}(2,m)$. It holds $\Pscc = \Pcone$ if and only if $U = \Usquare$, i.e. $U$ defines a parallelogram. 
\end{corollary}

\textit{Proof.}
If $U$ defines a parallelogram, we already know $\Pscc = \Pcone$.

Assume now, that $\Pscc = \Pcone$. From Theorem \ref{thm:VRepresentationOfclPconePlanar}, we know that 

    \begin{equation}
        \clPcone = \conv\left\{ \bigcup_{i : u_i \in \Utriangle} \{ e_i \} \cup \bigcup_{(i,j) : u_i \in \Usquare, u_j \in U_{\square,u_i} } \left\{ \frac{1}{2}(e_i + e_j) \right\} \right\}.
    \end{equation}
The set $\Utriangle$ must be empty, since we assume that the cone-volume set equals the subspace concentration polytope and every entry of an element contained in $\Pscc$ is less or equal than $\frac{1}{2}$. Hence, it holds $U = \Usquare$ and thus $\pos(\Usquare) = \R^2$. Therefore, $\Usquare$ contains at least $n + 1 = 3$ elements, \cite[Thoerem 4]{shephard1971diagrams}, and from Remark \ref{Remark:StrauctureUSquare} if follows $|\Usquare| = 4$. Hence, the set $U$ defines a parallelogram. 
\qed

\vspace{\baselineskip}

The vertex representation of the closure of the cone-volume set $C_\cv(U)$ implies a similar version of Theorem \ref{StancuSufficientConditions} i), at least for the convex hull of the cone-volume set. 

\begin{corollary}
\label{Corollary:coneVolumeSetEqualsHyperSimplex}
    For a set $U \in \mathcal{U}(2,m)$ is holds 
    \begin{equation*}
        \clPcone = \conv\{e_1,\dots,e_m\}
    \end{equation*}
    if and only if $U = \Utriangle$.
\end{corollary}

\textit{Proof.}
That is an immediate consequence of Theorem \ref{thm:VRepresentationOfclPconePlanar}.
\qed

\vspace{\baselineskip}

Coroallary \ref{Corollary:coneVolumeSetEqualsHyperSimplex} implies a weaker version of Theorem \ref{StancuSufficientConditions} i). Consider a set $U \in \mathcal{U}(2,m)$ that is in general position. It is easy to check that this implies $U = \Utriangle$ and thus the set $\clPcone$ is equal to $\conv(e_1,\dots,e_m)$.

We want to note that Corollary \ref{Corollary:coneVolumeSetEqualsHyperSimplex} can be generalized to arbitrary dimensions by generalizing the definition $\Utriangle$ to any dimension, which will be done in a different paper. However, the generalization of Theorem \ref{thm:VRepresentationOfclPconePlanar} to any dimensions seems non-trivial, since it is not clear what sets we need to consider in comparison to $\Usquare$.

\bibliographystyle{abbrvnat}     
 \bibliography{PhDBib}

\end{document}